\documentclass[a4paper,11pt]{article}
\usepackage{amsmath}
\usepackage{amssymb} 
\usepackage{theorem}
\usepackage{fullpage}
\usepackage[utf8x]{inputenc}
\usepackage{graphicx}

\newtheorem{theorem}{Theorem}[section]
\newtheorem{corollary}[theorem]{Corollary}
\theorembodyfont{\rm} 
\newtheorem{example}[theorem]{Example}
\newtheorem{remark}[theorem]{Remark}

\newenvironment{proof}{\noindent {\it Proof}.}{\hfill$\Box$}

\numberwithin{equation}{section}
\def\IN{\mathbb{N}}
\def\IR{\mathbb{R}}
\def\IS{\mathbb{S}}
\def\IZ{\mathbb{Z}}
\begin{document}
\title{Periodic orbits of large diameter for circle maps
\footnotetext{  }
\footnotetext{Proceedings of the American Mathematical Society, 138, No 9, 3211-3217, 2010.}
}
\author{Llu\'{\i}s Alsed\`a and Sylvie Ruette}

\date{}
\maketitle
\begin{abstract}
Let $f$ be a continuous circle map and let $F$ be a lifting of $f$. In
this note we study how the existence of a large orbit for $F$ affects
its set of periods. More precisely, we show that, if $F$ is of degree
$d\geq 1$ and has a periodic orbit of diameter larger than 1, then $F$
 has periodic points of period $n$ for all integers $n\geq 1$, and
thus so has $f$. We also give examples showing that this result does
not hold when the degree is non positive.
\end{abstract}

\section{Introduction}
One of the basic problems in topological dynamics in one dimension is
the characterization of the sets of periods of all periodic points.
This problem has its roots and motivation in Sharkovskii theorem
\cite{Sha}.
A lot of effort has been spent in generalizing Sharkovskii theorem for
more and more general classes of continuous self maps on trees, and
finally the characterization of the set of periods of general tree
maps is given in \cite{AJM3}.
While the set of periods of tree maps can be described with a finite
number of orderings, circle maps display new features. The set of
periods of a continuous circle map depends on the degree of the map
(see, e.g., \cite{ALM}). Consider a continuous map $f\colon\IS\to
\IS$, where $\IS=\IR/\IZ$, and $F$ a lifting of $f$, that is, a
continuous map $F\colon\IR\to\IR$ such that $f\circ\pi=\pi\circ F$,
where $\pi\colon \IR\to \IS$ is the canonical projection ($F$ is
uniquely defined up to the addition of an integer). The degree of $f$
(or $F$) is the integer $d\in\IZ$ such that $F(x+1)=F(x)+d$ for all
$x\in\IR$. If $|d|\geq 2$, then the set of periods is $\IN$ (the case
$\IN\setminus\{2\}$ is also possible when $d=-2$). If $d=0$ or $d=-1$,
then the possible sets of periods are ruled by Sharkovskii order, as
for continuous interval maps. The case $d=1$ is the most complex one
and requires the rotation theory. Let $F$ be a lifting of a degree 1
circle map $f$. The rotation number of a point $x\in\IR$ is
$\rho_{_F}(x)=\lim_{n\to+\infty}\frac{F^n(x)-x}{n}$, when the limit
exists. The set of all rotation numbers is a compact interval $[a,b]$,
and the set of periods of $f$ contains
\[
\{q\in\IN\mid \exists\, p\in\IZ,\ a<\frac{p}{q}<b\}.
\]
This comes from the knowledge of the set of periods of periodic points
with a given rotation number, which can be reduced from the study of
periods of points of rotation number 0.

In this note, we show that the set of periods of a lifting $F$ of a
circle map $f$ of degree $d\geq 1$ is $\IN$ if $F$ has a periodic
orbit of diameter larger than 1. This result obviously projects on the
circle: if such a periodic orbit exists for $F$ (for $f$, this means
that the periodic orbit ``spreads'' on more than one turn on the
circle), then the set of periods of $f$ is $\IN$. Our result improves
the well known fact that the set of periods of $f$ is $\IN$ for $d \ge
2$. Indeed, when a large orbit exists, it shows that there is a
subclass of orbits of $f$ (namely those which come from a true
periodic orbit of a lifting $F$) whose set of periods already contains
$\IN$. This study, in addition to its own interest, is mainly
motivated by the case $d=1$ because it might shed some light on the
characterization of the set of periods of maps of degree 1 on
topological graphs containing a loop. In particular the graph shaped
like the letter $\sigma$ (an interval glued to a circle). For liftings
of maps of the graph $\sigma$, it seems that periodic orbits of
rotation number 0 of ``large'' diameter force all periods greater than
or equal to 2. When the branching point of $\sigma$ is fixed, the
possible sets of periods are known \cite{LL}. On the other hand, a
rotation theory has been developed for continuous self maps on
topological graphs with a unique loop in \cite{AR}, and the rotation
set of a $\sigma$ map is studied in \cite{R9}, which is a first step
in the comprehension of the case of graph maps of degree 1.

\section{Statement and proof of the result}

Let $F\colon\IR\to \IR$ be a continuous map. A point $x\in\IR$ is
\emph{periodic} (for $F$) if there exists an integer $n\geq 1$ such
that $F^n(x)=x$. The \emph{period} of $x$ is the least integer $n$
with this property, that is, $F^n(x)=x$ and $F^i(x)\neq x$ for all
$1\leq i\leq n-1$. A \emph{periodic orbit} is the orbit of some
periodic point $x$, that is, $\{F^i(x)\mid i\geq 0\}$, which is a
finite set. A set $A\subset\IR$ is  \emph{$F$-invariant} if
$F(A)\subset A$. Clearly, the only non empty $F$-invariant subset of a
periodic orbit $P$ is $P$ itself.

\begin{remark}
Let $F$ be a lifting of a circle map $f\colon \IS\to \IS$. Then a
point $\pi(x)\in\IS$ is periodic for $f$ if and only if $x$ is
periodic (mod 1) for $F$, that is, $\exists n\geq 1, k\in\IZ, F^n(x)=
x+k$. If in addition $f$ is of degree 1, then $\rho_F(x)=k/n$, and
thus the periodic points of $F$ are exactly the periodic (mod 1)
points of rotation number 0.
\end{remark}

Now we state the main result of this note.

\begin{theorem}\label{prop:periodsF}
Let $F\colon\IR\to\IR$ be a continuous map which is the lifting of a
circle map of degree $d\geq 1$. Assume that $F$ has a periodic orbit
$P$ of period $n$ such that $\max P-\min P>1$. If $d=1$, then the
rotation interval of $F$ contains the interval
$[-\tfrac{1}{n},\tfrac{1}{n}]$ and, consequently, $F$ has periodic
points of all periods. If $d \ge 2$ then, $F$ also has periodic points
of all periods.
\end{theorem}

\begin{proof}
We consider separately the cases $d=1$ and $d \ge 2$.

Assume first that $d=1$. Set $p := \min P$ and let $k < n$ be
the positive integer such that $F^k(p) = \max P > p+1$. Let
\[
 F_u(x) := \sup\{f(y) \colon y \le x\}.
\]
From \cite[Proposition~3.7.7(d)]{ALM} it follows that $F_u$ is
continuous, non-decreasing and has degree one (that is, $F_u(x+1) =
F_u(x)+1$ for every $x \in \IR$). Moreover, if we use
\cite[Proposition~3.7.7(a)]{ALM} and the fact that $F_u$ is non decreasing,
we get $F_u^i(x)\ge F^i(x)$ for all $x\in \IR$ and $i\geq 1$, and hence
$F_u^k(p) > p+1$.

Assume that $F^{k\ell}_u(p) > p+\ell$ for some $\ell\in\IN$. Then, by
\cite[Proposition~3.1.7(c)]{ALM},
\[
F^{k(\ell+1)}_u(p) = F^k_u(F^{k\ell}_u(p)) \ge F^k_u(p+\ell) =
F^k_u(p) +\ell > p + (\ell + 1).
\]
Hence, $F^{k\ell}_u(p) > p+\ell$ for every $\ell > 0$ and,
consequently,
\[
\limsup_{j\to+\infty}\frac{F_u^j(p)-p}{j} \ge
\limsup_{\ell\to+\infty}\frac{F_u^{k\ell}(p)-p}{k\ell} \ge
\tfrac{1}{k} > \tfrac{1}{n}.
\]
On the other hand, since $F_u$ is non-decreasing, \cite[Theorem~1]{RT}
implies that $\rho_{_{F_u}}(x)$ exists for each $x \in \IR$ and is
independent of the choice of the point $x$. This number is called
the \emph{rotation number of $F_u$} and denoted by $\rho(F_u)$.
From above it follows that $\rho(F_u) = \rho_{_{F_u}}(p) >
\tfrac{1}{n}$. Then, in view of \cite[Theorem~3.7.20(a)]{ALM} it
follows that the right endpoint of the rotation interval of $F$ is
larger than $\tfrac{1}{n}$. In a similar way (using $\max P$ instead
of $\min P$) it follows that the left endpoint of the rotation
interval of $F$ is smaller than $-\tfrac{1}{n}$. Thus, the theorem in
the case $d=1$ follows from \cite[Lemma~3.9.1]{ALM}.

Now we consider the case $d\ge 2$. As above we set $p := \min P$, and
$q := \max P > p+1$. Since the $F$-orbit of $P$ is periodic, $F^j(p)
\ge p$ for every $j \ge 0$. So, by \cite[Proposition~3.1.7(c)]{ALM},
the sequence $\{F^j(p+1)\}_{j=0}^\infty$ is contained in $(p,+\infty)$
and diverges to $+\infty$ (in particular, $F^j(p+1) \ne q$ for every
$j$). Since $p+1 < q$ there exists $m \ge 0$ such that $p <
r:=F^m(p+1) < q$ but $F^j(r) > q$ for every $j > 0$. Since $P
\not\subset [r,+\infty)$, there exists $s \in P$ such that $q \ge s >
r$ but $F(s) < r$. Set $I=[r,s]$ and $J=[s,F(r)]$. Then $F(I)\supset
I\cup J$ and $F(J)\supset I$. It is well known that in this situation,
there exist periodic points of period $\ell$ for every integer
$\ell\geq 1$. To give a precise proof, we use \cite[Corollary
1.2.8]{ALM}: for $\ell=1$, we use $F(I)\supset I$; and for all
$\ell\geq 2$, we get that there exists $x\in J$ such that
$F^\ell(x)=x$ and $F^i(x)\in I$ for all $1\leq i\leq \ell-1$, which
implies that $x$ is periodic of period $\ell$ (indeed, if $F^i(x)=x$
for some $1\leq i \leq \ell-1$ then $x\in I\cap J$, which is
impossible because $F(s)\not\in I\cup J$). This ends the proof of the
theorem.
\end{proof}

\begin{remark}
A simple generalization of the above theorem and its proof for the
case $d=1$ is the following. Assume that $F$ has periodic orbits
$P_1, P_2, \dots, P_j$ such that the set $\bigcup_{i=1}^j
\langle P_i \rangle$ is connected and has diameter larger than one,
where $\langle P_i \rangle$ denotes the \emph{convex hull of $P_i$}
(that is, the smallest closed interval containing $P_i$).
By ordering the periodic orbits (and possibly withdrawing some of them), we may
assume that $\min P_{i+1}\le \max P_i$ for all $1\le i\le j-1$.
Let $|P_i|$ denote the period of $P_i$.
For each $1\le i\le j$, there exists a positive integer $k_i< |P_i|$
such that $F^{k_i}(\min P_i)=\max P_i$. Let $p:=\min P_1=\min
\bigcup_{i=1}^j P_i$. Using the facts that $F_u$ is non decreasing and
$F_u^k(x)\geq F^k(x)$ for all $x\in \IR$ and all $k\ge 1$ (see
the proof of Theorem~\ref{prop:periodsF}), we get
\begin{eqnarray*}
F_u^{k_1}(p)&\ge&\max P_1\ge \min P_2\\
F_u^{k_1+k_2}(p)&\ge&F_u^{k_2}(\min P_2)\ge \max P_2\ge \min P_3\\
&\vdots&\\
F_u^{k_1+\ldots+k_j}(p)&\ge&\max P_j=\max \bigcup_{i=1}^j P_i>p+1.
\end{eqnarray*}
Then, in a
similar way as in the proof of Theorem~\ref{prop:periodsF}, it is
possible to show that, for every $\ell > 0$,
\[
 F^{m\ell}_u(p) > p+\ell
\]
where $m=k_1+\cdots k_j <n:=\sum_{i=1}^j |P_i|$. Consequently,
the rotation interval of $F$ contains the non-degenerate interval
$[-\tfrac{1}{n}, \tfrac{1}{n}]$. Thus, $F$ has periodic
points of all periods.
\end{remark}

The next corollary is a straightforward consequence of
Theorem~\ref{prop:periodsF}.

\begin{corollary}
Let $f\colon \IS\to \IS$ be a continuous circle map of degree $d\geq
1$ and $F$ a lifting of $f$. If there exists a periodic orbit $P$ for
$F$ such that $\max P-\min P>1$ then $f$ has periodic points of all
periods.
\end{corollary}

The conclusion of Theorem~\ref{prop:periodsF} does not hold when
the degree $d$ is non positive. For $d=-1$, $F(x)=-x$ gives a trivial
counter-example. The cases $d=0$ and $d\leq -2$ are treated
in Examples~\ref{ex:0} and \ref{ex:-d}, respectively.

\begin{example}\label{ex:-d}
Let $d$ be an integer, $d\geq 2$, and let
$\widetilde{F}\colon [0,1]\to\IR$ be the map defined by (see
Figure~\ref{fig:Fd}):
\[
 \widetilde{F}(x):=\begin{cases}
        (3-4d)x & \text{if $x \in [0,1/4]$,}\\
        (2d-3)x+\frac{3(1-d)}{2} & \text{if $x \in [1/4,3/4]$,}\\
        (3-4d)x+3(d-1) & \text{if $x \in [3/4,1]$.}
       \end{cases}
\]
Observe that
\begin{align*}
& (2d-3)\tfrac{1}{4}+\tfrac{3(1-d)}{2}
   = \tfrac{3}{4} - d
   = (3-4d)\tfrac{1}{4},\quad\text{and}\\
& (2d-3)\tfrac{3}{4}+\tfrac{3(1-d)}{2}
   = - \tfrac{3}{4}
   = (3-4d)\tfrac{3}{4}+3(d-1).
\end{align*}
Therefore, $\widetilde{F}$ is continuous, $\widetilde{F}(0)=0,$
$\widetilde{F}(\tfrac{1}{4})=\tfrac{3}{4} - d,$
$\widetilde{F}(\tfrac{3}{4}) = -\tfrac{3}{4}$ and
$\widetilde{F}(1)=-d.$ Observe that $\widetilde{F}(x) = -x$ for
$x\in\{0,\tfrac{3}{4}\}$ and $\widetilde{F}(x) < -x$ for $x\in
[0,1]\setminus\{0,\tfrac{3}{4}\}$. Moreover, it is a straightforward
computation to show that $\widetilde{F}(x) + d =-\widetilde{F}(1-x)$
by considering separately the cases $x \in [0,1/4] \cup [3/4,1]$ and
$x \in [1/4,3/4]$.

Now consider the map $F\colon \IR\to\IR$ defined by $F(x) :=
\widetilde{F}(x - \lfloor x \rfloor) -d\lfloor x \rfloor$, where
$\lfloor x \rfloor$ denotes the integer part of $x$ (see
Figure~\ref{fig:Fd}).
\begin{figure}[htb]
\begin{center}
\includegraphics[height=0.64\textheight]{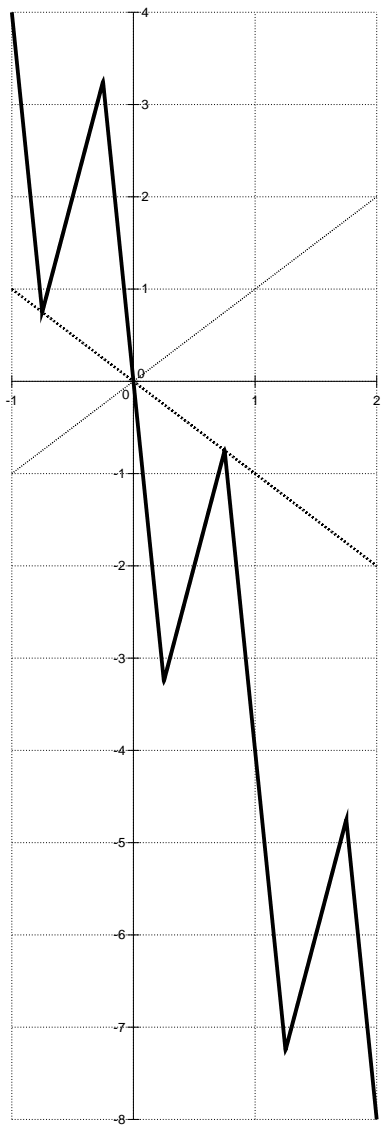}
\end{center}
\caption{The graphs of the map $F$ with $d = 4$ and $y=-x$
(dotted line) in the interval $[-1,2].$
\label{fig:Fd}}
\end{figure}

Clearly $F$ is a lifting of a continuous map of the circle of degree
$-d$ (in particular, $F(x+k) = F(x) -dk$ for every $x \in \IR$ and $k
\in \IZ$). Moreover, $F$ is odd. To see it take $x \in \IR$ and write
$x = \lfloor x \rfloor + \widetilde{x}$ with $\widetilde{x} \in
[0,1)$. Then,
\begin{align*}
 - F(-x)
 & = - F(- \widetilde{x} -\lfloor x \rfloor)
   = - F(-( \lfloor x \rfloor + 1) + (1 - \widetilde{x}))\\
 & = -F(1-\widetilde{x}) -d(\lfloor x \rfloor + 1)
   = -\widetilde{F}(1-\widetilde{x}) -d(\lfloor x \rfloor + 1)\\
 & = \widetilde{F}(\widetilde{x})+d -d(\lfloor x \rfloor + 1)
   = F(\widetilde{x}) -d\lfloor x \rfloor
   = F(\widetilde{x} + \lfloor x \rfloor) = F(x).
\end{align*}

From above it follows that $F(0) = 0,$ $F(\tfrac{3}{4}) =
-\tfrac{3}{4}$ and $F(-\tfrac{3}{4}) = \tfrac{3}{4}$. Hence, $0$ is a
fixed point of $F$ whereas $\{-\tfrac{3}{4},\tfrac{3}{4}\}$ is a
periodic orbit of $F$ of period 2 with diameter larger than one. To
end this example we will show that $F$ has no other periodic
points.

We claim that $|F(x)| > |x|$ for all $x \in \IR
\setminus\{-\tfrac{3}{4},0,\tfrac{3}{4}\}$. When $x \in [0,1]$ this
amounts to showing that $F(x) < -x$ whenever $x \notin
\{0,\tfrac{3}{4}\}$ and this follows from our remarks on
$\widetilde{F}$. When $x \ge 1$ we have $\lfloor x \rfloor +1 > x \ge
\lfloor x \rfloor \ge 1$ and, hence,
\[
 F(x) = F(x-\lfloor x \rfloor) -d\lfloor x \rfloor \le
- \lfloor x \rfloor - 1 < -x.
\]
The case $x < 0$ follows from the fact that $F$ is odd. This ends the
proof of the claim.

From the above claim it follows that if $x \in \IR$ is not
a preimage of $0$ or $\tfrac{3}{4}$ under some iterate of $F$, then $|x| <
|F(x)| < |F^2(x)| < \cdots$ and thus it cannot be periodic. Hence, $F$
has no periodic points other than $\{-\tfrac{3}{4},0,\tfrac{3}{4}\}$.
\end{example}

\begin{example}\label{ex:0}
We define a (continuous) lifting of circle map of degree 0
$F\colon\IR\to\IR$ as follows. First we choose $p \ge 3$ odd and
points
$x_0,x_1,\dots,x_p$ and $z_0,z_1,\dots,z_p$ in $\IR$ such that
\begin{multline*}
x_0 < z_0-1 < x_p < z_{p-1} < x_{p-2} < z_{p-3} < \dots < x_3 <
          z_2 < x_1 < \\
z_1 < x_2 < z_3 < \dots x_{p-3} < z_{p-2} < x_{p-1} <
      z_p < x_0+1 < z_0.
\end{multline*}
Set $P := \{x_0,x_1,\dots,x_p,z_0,z_1,\dots,z_p\}$ and $\widetilde{P}
:= P \cup \{z_0-1, x_0+1\}$. Then we define $F$ so that $F(x_i) =
x_{i+1}$ and $F(z_i) = z_{i+1}$ for $i=0,1,\dots,p-1$, $F(x_p) = z_0$,
$F(z_p) = x_0$, $F$ is affine in the closure of every connected
component of $[x_0,z_0]\setminus \widetilde{P}$ and furthermore we
impose that $F(x+1) = F(x)$ for every $x \in \IR$ (in particular
$F(z_0-1) = F(z_0) = z_1$ and $F(x_0+1) = F(x_0) = x_1$) (see
Figure~\ref{exdeg0} for an example with $p=3$).
\begin{figure}[htb]
\centerline{\includegraphics{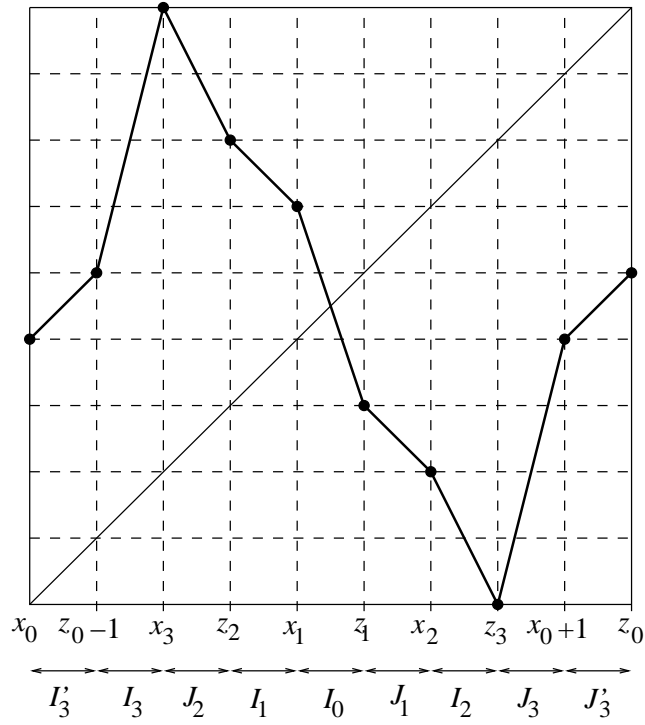}}
\caption{Graph of $F$ for $p=3$.
\label{exdeg0}}
\end{figure}

Clearly, the above conditions define a continuous function from $\IR$
to itself that is the lifting of a circle map of degree 0.
Moreover, $P$ is a periodic orbit of $F$ of period $2p+2$  and this
orbit is large since $\max P = z_0 > x_0+1 = \min P +1$. We will show
that $F$ has no periodic orbits of period $3,5,\dots,p$. Thus, $F$
does not have periodic points of all periods.

To show our claim we will compute the Markov graph of the map $F$ and
show that it has no loops of the specified length. We observe that,
by definition, $F(\IR) = [x_0,z_0]$. So we only have to consider the
graph on the finitely many vertexes contained in $[x_0,z_0]$.
To this end, we define the intervals $I_0 := [x_1,z_1]$,  $I_i :=
\langle x_i,z_{i+1}\rangle$ and $J_i= \langle z_i,x_{i+1}\rangle$ for
$i=1,2,\dots,p-1$ (where $\langle a,b\rangle$ denotes either $[a,b]$
or
$[b,a]$ depending on the order of $a,b$),
$I_p := [z_0-1, x_p]$, $J_p := [z_p,x_0+1]$, $I'_p
:= [x_0, z_0-1]$, $J'_p := [x_0+1,z_0]$. Then $F$ is a Markov map with
respect
to this partition and its Markov graph has exactly the following
arrows:

\begin{itemize}
 \item $I_0 \longrightarrow I_0$,
 \item $I_0 \longrightarrow I_1 \longrightarrow I_2 \longrightarrow
       \dots \longrightarrow I_{p-1}
       \begin{smallmatrix} \nearrow\\ \searrow\end{smallmatrix}
       \begin{array}{l} I'_p \longrightarrow I_0 \\ I_p \end{array}$,
\item $I_0 \longrightarrow J_1 \longrightarrow J_2 \longrightarrow
\dots \longrightarrow J_{p-1}
      \begin{smallmatrix} \nearrow\\ \searrow\end{smallmatrix}
       \begin{array}{l} J'_p \longrightarrow I_0 \\ J_p \end{array}$,
\item $I_p \longrightarrow K$ for all $K\in\{J_1,J_3,\dots,J_p,J'_p,
I_2,I_4,
\dots,I_{p-1}\}$,
\item $J_p\longrightarrow K$ for all $K\in\{I_1,I_3,\dots,I_p,I'_p,
J_2,J_4,\dots,J_{p-1}\}$.
\end{itemize}
By direct inspection one can see that in the above graph any loop
contains either $I_0$, $I_p$ or $J_p$. Moreover the loops not
containing $I_0$ are all of even length. The shorter simple
loops of odd length greater than 1 are exactly
the following four loops of length $p+2$:
\begin{itemize}
\item $I_0 \longrightarrow I_1 \longrightarrow I_2 \longrightarrow
       \dots \longrightarrow I_{p-1} \longrightarrow I'_p
       \longrightarrow I_0 \longrightarrow I_0$,
\item $I_0 \longrightarrow I_1 \longrightarrow I_2 \longrightarrow
       \dots \longrightarrow I_{p-1} \longrightarrow I_p
        \longrightarrow J'_p \longrightarrow I_0$,
\item $I_0 \longrightarrow J_1 \longrightarrow J_2 \longrightarrow
       \dots \longrightarrow J_{p-1} \longrightarrow J'_p
       \longrightarrow I_0 \longrightarrow I_0$,
\item $I_0 \longrightarrow J_1 \longrightarrow J_2 \longrightarrow
       \dots \longrightarrow J_{p-1} \longrightarrow J_p
        \longrightarrow I'_p \longrightarrow I_0$.
\end{itemize}
Consequently the Markov graph of $F$ has no loops of lengths
$3,5,\dots,p$ and, by \cite[Lemma 1.2.12]{ALM}, the map $F$ cannot have
periodic points of any of these periods.
\end{example}

\section*{Acknowledgments}
We thank the anonymous referee for detailed and clever comments that helped
us improve this article.

\noindent
{\scshape Llu\'{\i}s Alsed\`a}\footnote{Partially supported by MEC grant number MTM2008-01486.} -- Departament de Matem\`{a}tiques,
Edifici Cc, Universitat Aut\`{o}noma de Barcelona,
08913 Cerdanyola del Vall\`es, Barcelona,
Spain\\ 
{\it E-mail address:} {\tt alseda@mat.uab.cat}

\medskip\noindent
{\scshape Sylvie Ruette} -- Laboratoire de Math\'ematiques,
CNRS UMR 8628, B\^atiment 425,
Universit\'e Paris-Sud 11,
91405 Orsay cedex,
France\\
{\it Email address:} {\tt sylvie.ruette@math.u-psud.fr}

\end{document}